\documentclass[a4paper,10pt]{amsart}
\usepackage{amsmath,amssymb}
\usepackage{amscd}
\newtheorem{thm}{Theorem}[section]
\newtheorem{prop}[thm]{Proposition}
\newtheorem{lemma}[thm]{Lemma}
\newtheorem{cor}[thm]{Corollary}

\theoremstyle{remark}
\newtheorem{remark}[thm]{Remark}
\newcommand{\id}{{\rm{id}}}

\newcommand{\BN}{\mathbf N}

\newcommand{\BC}{\mathbf C}

\newcommand{\BB}{\mathbf B}

\newcommand{\la}{\langle}
\newcommand{\ra}{\rangle}

\newcommand{\Equi}{{\rm{Equi}}}

\newtheorem{Def}{Definition}[section]

\title{Bimodule maps from a unital $C^*$-algebra to its $C^*$-subalgebra and strong
Morita equivalence}
\author{Kazunori Kodaka}
\address{Department of Mathematical Sciences, Faculty of Science, Ryukyu
\endgraf
University, Nishihara-cho, Okinawa, 903-0213, Japan}
\address{\sl{E-mail address}: \rm{kodaka@math.u-ryukyu.ac.jp}}


\begin{document}
%\maketitle
\begin{abstract}
Let $A\subset C$ and $B\subset D$ be unital inclusions of unital $C^*$-algebras.
Let ${}_A \BB_A (C, A)$ (resp. ${}_B \BB_B (D, B)$) be the space of all bounded $A$-bimodule
(resp. $B$-bimodule) linear maps from $C$ (resp. $D$) to $A$ (resp. $B$). We suppose that
$A\subset C$ and $B\subset D$ are strongly Morita equivalent. We shall show that there is an isometric
isomorphism $f$ of ${}_A \BB_A (C, A)$ onto ${}_B \BB_B (D, B)$ and we shall study on basic properties about
$f$.
\end{abstract}

\maketitle

\section{Introduction}\label{sec:intro} Let $A\subset C$ and $B\subset D$ be unital inclusions of
unital $C^*$-algebras. We suppose that they are strongly Morita equivalent with respect to a
$C-D$-equivalence bimodule $Y$ an its closed subspace $X$.
In this paper, we shall
define an isometric isomorphism $f$ of ${}_A \BB_A (C, A)$ onto ${}_B \BB_B (D, B)$ induced by
$Y$ and $X$, where ${}_A \BB_A (C, A)$
(resp. ${}_B \BB_B (D, B)$) is the space of all bounded $A$-bimodule (resp. $B$-bimodule) linear maps from
$C$ (resp. $D$) to $A$ (resp. $B$). We shall show that the above isometric isomorphism $f$ can be
constructed in the same way as in \cite [Section 2]{KT4:morita}. Using the above result, we study on the basic
properties about $f$. Especially we shall give the following result: If $\phi$ is an element in ${}_A \BB_A (C, A)$
having a quasi-basis defined in Watatani \cite [Definition 1.11.1]{Watatani:index}, then $f(\phi)$ is also an
element in ${}_B \BB_B (D, B)$ having a quasi-basis and there is an isomorphism of $\pi$ of $A' \cap C$ onto
$B' \cap D$ such that
$$
\theta^{f(\phi)}=\pi\circ\theta^{\phi}\circ\pi^{-1} ,
$$
where $\theta^{\phi}$ and $\theta^{f(\phi)}$ are the modular automorphisms for $\phi$ and $f(\phi)$,
respectively which are defined in \cite [Definition 1.11.2]{Watatani:index}.
We note that the isometric isomorphism $f$ of ${}_A \BB_A (C, A)$ onto ${}_B \BB_B (D, B)$ depends on the
choice of a $C-D$-equivalence bimodule and its closed subspace $X$. In the last section, we shall
discuss the relation between $f$ and the pair $(X, Y)$.
\par
For an algebra $A$, we denote by $1_A$ and $\id_A$ the unit element in $A$ and the identity
map on $A$, respectively. If no confusion arises, we denote them by $1$ and $\id$, respectively.
For each $n\in\BN$, we denote by $M_n (\BC)$ the $n\times n$-matrix algebra over $\BC$ and $I_n$
denotes the unit element in $M_n (\BC)$. Also, we denote by $M_n (A)$ the $n\times n$-matrix algebra
over $A$ and we identify $M_n(A)$ with $A\otimes M_n (\BC)$ for any $n\in \BN$.
\par
Let $A$ and $B$ be $C^*$-algebras. Let $X$ be an $A-B$-equivalence bimodule.
For any $a\in A$, $b\in B$, $x\in X$,
we denote by $a\cdot x$ the left $A$-action on $X$ and by $x\cdot b$ the right $B$-action on $X$, respectively.
Let ${}_A \BB(X)$ be the $C^*$-algebra of all adjointable left $A$-linear operators on $X$
and we identify ${}_A \BB(X)$ with $B$. Similarly we define $\BB_B (X)$ and we identify $\BB_B (X)$ with $A$.

\section{Construction}\label{sec:con} Let $A\subset C$ and $B\subset D$ be
unital inclusions of unital $C^*$-algebras.
Let ${}_A \BB_A (C, A)$ and ${}_B \BB_B (D, B)$ be as in Introduction.
We suppose that $A\subset C$ and
$B\subset D$ are strongly Morita equivalent with respect to a $C-D$-equivalence bimodule $Y$ and
its closed subspace $X$. We construct an isometric isomorphism of ${}_A \BB_A (C, A)$ onto ${}_B \BB_B (D, B)$.
Let $\phi\in {}_A \BB_A (C, A)$. In the same way as in the proof of \cite [Lemma 3.4]{KT4:morita}, we define the linear
map $\tau$ from $Y$ to $X$ by
$$
{}_A \la \tau(y) \, , \, x \ra =\phi( \, {}_C \la y, x \ra)
$$
for any $x\in X$, $y\in Y$.

\begin{lemma}\label{lem:con1} With the above notation, $\tau$ satisfies the following conditions:
\newline
$(1)$ $\tau(c\cdot x)=\phi(c)\cdot x$,
\newline
$(2)$ $\tau(a\cdot y)=a\cdot \tau(y)$,
\newline
$(3)$ ${}_A \la \tau(y) \, , \, x \ra=\phi( \, {}_C \la y, x \ra)$
\newline
for any $a\in A$, $c\in C$, $x\in X$, $y\in Y$.
Also, $\tau$ is bounded and $||\tau||\leq ||\phi||$.
Furthermore, $\tau$ is the unique linear map from $Y$ to $X$ satisfying Condition $(3)$.
\end{lemma}
\begin{proof} Except for the uniqueness of $\tau$, we can prove this lemma in the same way as in the proof
of \cite [Lemma 3.4 and Remark 3.3(ii)]{KT4:morita}. Indeed, by the definition of $\tau$, clearly $\tau$ satisfies
Condition (3). For any $x, z\in X$, $c\in C$,
$$
{}_A \la \tau(c\cdot x) \, , \, z \ra =\phi( \, {}_C \la c\cdot x \, , \, z \ra)=\phi(c\, {}_A \la x, z \ra)=\phi(c)\, {}_A \la x, z \ra
={}_A \la \phi(c)\cdot x \, , \, z \ra .
$$
Hence $\tau(c\cdot x)=\phi(c)\cdot x$ for any $x\in X$, $c\in C$. Also, for any $a\in A$, $y\in Y$, $z\in X$,
$$
{}_A \la \tau(a\cdot y) \, , \, z \ra =\phi( \, {}_C \la a\cdot y \, , \, z \ra)=a\phi(\, {}_C \la y, z \ra)
=a\, {}_A \la \tau(y) \, , \, z \ra =\, {}_A \la a\cdot \tau(y) \, , \, x \ra .
$$
Hence $\tau(a\cdot y)=a\cdot \tau(y)$ fo any $a\in A$, $y\in Y$. By Raeburn and Williams \cite [proof of Lemma 2.8]
{RW:continuous},
\begin{align*}
||\tau(y)|| & =\sup \{||\, {}_A \la \tau(y) \, , \, z \ra || \, \, | \, \, ||z||\leq 1, \, z\in X \} \\
& =\sup \{ \, ||\phi(\, {}_C \la y , z \ra)|| \, \, | \, \, ||z||\leq 1, \, z\in X \} \\
& \leq \sup\{ \, ||\phi|| \, ||y|| \, ||z|| \, \, | \, \, ||z||\leq 1, \, z\in X \} \\
& \leq ||\phi|| \, ||y|| .
\end{align*}
Thus $\tau$ is bounded and $||\tau||\leq ||\phi||$. Furthermore,
let $\tau'$ be a linear map from $Y$ to $X$ satisfying Condition (3). Then
for any $x\in X$, $y\in Y$,
$$
{}_A \la \tau(y) \, , \, x \ra =\phi( \, {}_C \la y, x \ra)={}_A \la \tau' (y) \, , \, x \ra .
$$
Hence $\tau(y)=\tau' (y)$ for any $y\in Y$. Thus $\tau$ is unique.
\end{proof}

\begin{lemma}\label{lem:con2} With the above notation, $\tau(y\cdot b)=\tau(y)\cdot b$ for any $b\in B$, $y\in Y$.
\end{lemma}
\begin{proof} This can be proved in the same way as in the proof of \cite [Lemma 3.5]{KT4:morita}.
Indeed, for any $x, z\in X$, $y\in Y$,
$$
\tau(y\cdot \la x, z \ra_B )=\tau(\, {}_C \la y , x \ra\cdot z)=\phi(\, {}_C \la y, x \ra )\cdot z
={}_A \la \tau (y) \, , \, x \ra \cdot z =\tau(y)\cdot \la x, z \ra_B .
$$
Since $X$ is full with the right $B$-valued inner product, we obtain the conclusion.
\end{proof}

Let $\psi$ be the linear map from $D$ to $B$ defined by
$$
x\cdot \psi(d)=\tau(x\cdot d)
$$
for any $d\in D$, $x\in X$, where we identify ${}_A \BB (X)$ with $B$ as $C^*$-algebras.

\begin{lemma}\label{lem:con3} With the above notation, $\psi$ is a linear map from $D$ to $B$
satisfying the following conditions:
\newline
$(1)$ $\tau(x\cdot d)=x\cdot \psi(d)$,
\newline
$(2)$ $\psi(\la x, y \ra_D )=\la x \, , \, \tau(d) \ra_B$
\newline
for any $d\in D$, $x\in X$, $y\in Y$.
Also, $\psi$ is a bounded $B$-bimodule map from $D$ to $B$ with
$||\psi||\leq ||\tau||$. Furthermore, $\psi$ is the unique linear map from $D$ to $B$ satisfying Condition $(1)$.
\end{lemma}
\begin{proof} We can prove this lemma in the same way as in the proof of \cite [Proposition 3.6]{KT4:morita}.
Indeed, by the definition of $\psi$, $\psi$ satisfies Condition (1). Also, for any $x, z\in X$, $y\in Y$,
$$
z\cdot \psi(\la x, y \ra_D )=\tau(z\cdot \la x, y \ra_D )=\tau({}_A \la z, x \ra\cdot y )
={}_A \la z, x \ra \cdot \tau(y)=z\cdot \la x\, , \, \tau(y) \ra_B .
$$
Hence $\psi(\la x, y \ra_D )=\la x \, , \, \tau(y) \ra_B$ for any $x\in X$, $y\in Y$.
For any $d\in D$,
\begin{align*}
||\psi(d)|| & =\sup \{\, ||x\cdot \psi(d) || \, | \, ||x||\leq 1 \, , \, x\in X \} \\
& =\sup \{ \, ||\tau(x\cdot d) || \, | \, ||x||\leq1 \, , \, x\in X  \} \\
& \leq \sup\{ \, ||\tau||\,||x||\,||d|| \, | \, ||x||\leq1 \, , \, x\in X  \} \\
& =||\tau|| \, ||d|| .
\end{align*}
Thus $\psi$ is bounded and $||\psi||\leq||\tau||$. Next, we show that $\psi$ is a $B$-bimodule map from
$D$ to $B$. It suffices to show that
$$
\psi(bd)=b\psi(d) \, , \, \psi(db)=\psi(d)b
$$
for any $b\in B$, $d\in D$. For any $b\in B$, $d\in D$, $x\in X$,
$$
x\cdot \psi(bd)=\tau(x\cdot bd)=\tau((x\cdot b)\cdot d)=(x\cdot b)\cdot\psi(d)=x\cdot b\psi(d)
$$
since $x\cdot b\in X$. Hence $\psi(bd)=b\psi(d)$. Also,
$$
x\cdot \psi(db)=\tau(x\cdot db)=\tau(x\cdot d)\cdot b=x\cdot \psi(d)b
$$
by Lemma \ref {lem:con2}. Let $\psi'$ be a linear map from $D$ to $B$ satisfying Condition (1).
Then for any $x\in X$, $d\in D$,
$$
x\cdot \psi(d)=\tau(x\cdot d)=x\cdot \psi' (d) .
$$
Hence $\psi(d)=\psi' (d)$ for any $d\in D$. Therefore, we obtain the conclusion.
\end{proof}

\begin{prop}\label{prop:con4} Let $A\subset C$ and $B\subset D$ be unital inclusions of unital
$C^*$-algebras. We suppose that $A\subset C$ and $B\subset D$ are strongly Morita equivalent with
respect to a $C-D$-equivalence bimodule $Y$ and its closed subspace $X$. Let $\phi$ be any element
in ${}_A \BB_A (C, A)$. Then there are the unique linear map $\tau$ from $Y$ to $X$ and the
unique element $\psi$ in ${}_B \BB_B (D, B)$ satisfying the following conditions:
\newline
$(1)$ $\tau(c\cdot x)=\phi(c)\cdot x$,
\newline
$(2)$ $\tau(a\cdot y)=a\cdot \tau(y)$,
\newline
$(3)$ ${}_A \la \tau(y) \, , \, x \ra =\phi( \, {}_C \la y, x \ra )$,
\newline
$(4)$ $\tau(x\cdot d)=x\cdot \psi(d)$,
\newline
$(5)$ $\tau(y\cdot b)=\tau(y)\cdot b$,
\newline
$(6)$ $\psi(\la x, y \ra_D )=\la x \, , \, \tau(y) \ra_B$
\newline
for any $a\in A$, $b\in B$, $c\in C$, $d\in D$, $x\in X$, $y\in Y$.
Furthermore, $||\psi||\leq ||\tau|| \leq ||\phi||$. Also, for any element $\psi\in {}_B \BB_B (D, B)$, we
have the same results as above. 
\end{prop}
\begin{proof} This is immediate by Lemmas \ref{lem:con1}, \ref{lem:con2} and \ref{lem:con3}.
\end{proof}

For any element $\phi\in {}_A \BB_A (C, A)$, there are the unique element $\psi\in {}_B \BB_B (D, B)$
and the unique linear map $\tau:Y\to X$ satisfying Conditions (1)-(6) in Proposition \ref{prop:con4}.
We denote by $f$ the map from $\phi\in {}_A \BB_A (C, A)$ to the above $\psi\in {}_A \BB_B (D, B)$.
Then we have the
following theorem:

\begin{thm}\label{thm:con5} Let $A\subset C$ and $B\subset D$ be unital inclusions of unital $C^*$-algebras.
We suppose that $A\subset C$ and $B\subset D$ are strongly Morita equivalent. Then there is an isometric
isomorphism $f$ of ${}_A \BB_A (C, A)$ onto ${}_B \BB_B (D, B)$.
\end{thm}
\begin{proof} By the definition of $f$ and Proposition \ref{prop:con4},
$f$ is a linear map from ${}_A \BB_A (C, A)$ to ${}_B \BB_B (D, B)$.
Also, we can see that $f$ is a bijective isometric map from ${}_A \BB_A (C, A)$ onto ${}_B \BB_B (D, B)$
by Proposition \ref{prop:con4}.
\end{proof}

\begin{lemma}\label{lem:con6} With the above notation, let $\phi$ be any element in ${}_A \BB_A (C, A)$.
Then $f(\phi)$ is the unique linear map from $D$ to $B$ satisfying that
$$
{}_A \la x\cdot f(\phi)(d) \, , \, z \ra=\phi(\,{}_C \la x\cdot d \, , \, z \ra)
$$
for any $d\in D$, $x, z\in X$.
\end{lemma}
\begin{proof} Let $\psi$ be another linear map from $D$ to $B$ satisfying that
$$
{}_A \la x\cdot \psi(d) \, , \, z \ra=\phi(\,{}_C \la x\cdot d \, , \, z \ra)
$$
for any $d\in D$, $x, z\in X$. Then for any $z\in X$,
$$
{}_A \la x\cdot f(\phi)(d) \, , \, z \ra={}_A \la x\cdot \psi(d) \, , \, z \ra
$$
Hence $f(\phi)(d)=\psi(d)$ for any $d\in D$. Therefore $f(\phi)=\psi$.
\end{proof}

\section{Matrix algebras over a unital $C^*$-algebra}\label{sec:MA}
Let $A\subset C$ and $B\subset D$ be unital inclusions of unital
$C^*$-algebras. We suppose that $A\subset C$ and $B\subset D$ are strongly Morita equivalent with
respect to a $C-D$-equivalence bimodule $Y$ and its closed subspace $X$. By \cite [Section 2]{KT4:morita},
there are a positive integer $n$ and a full projection $p\in M_n (A)$ such that
$$
B\cong pM_n (A)p \, , \, D\cong pM_n (C)p
$$
as $C^*$-algebras and such that
$$
X\cong (1\otimes e)M_n (A)p \, , \, Y\cong (1\otimes e)M_n (C)p
$$
as $A-B$-equivalence bimodules and $C-D$-equivalence bimodules, respectively, where
$$
e=\begin{bmatrix} 1 & 0 & \cdots & 0 \\
0 & 0 & \cdots & 0 \\
\vdots & \vdots & \ddots & \vdots \\
0 & 0 & \cdots & 0 \end{bmatrix}\in M_n (\BC)
$$
and we identify $A$, $C$ and $B$, $D$ with $(1\otimes e)M_n (A)(1\otimes e)$, $(1\otimes e)M_n (C)(1\otimes e)$
and $pM_n (A)p$, $pM_n (C)p$, respectively. We denote the above isomorphisms by
\begin{align*}
\Psi_B : & \, B\to pM_n (A)p, \\
\Psi_D : & \, D\to pM_n (C)p, \\
\Psi_X : & \, X\to (1\otimes e)M_n (A)p, \\
\Psi_Y : & \, Y\to (1\otimes e)M_n (C)p,
\end{align*}
respectively. In this section, we shall construct a map from ${}_A \BB_A (C, A)$ to
the space of all $pM_n (A)p$-bimodule maps,
${}_{pM_n (A)p} \BB_{pM_n (A)p}(pM_n (C)p \, , \, pM_n (A)p)$. Let $\phi\in {}_A \BB_A (C, A)$.
Let $\psi$ be the map from $M_n (C)$ to $M_n (A)$ defined by
$$
\psi(x)=(\phi\otimes\id)(x)
$$
for any $x\in M_n (C)$. Since $\psi(p)=p$, by easy computations, $\psi$ can be regarded
as an element in ${}_{pM_n (A)p} \BB_{pM_n (A)p}(pM_n (C)p \, , \, pM_n (A)p)$. 
We denote by $F$ the map from $\phi\in {}_A \BB_A (C, A)$ to the above
$\psi\in{}_{pM_n (A)p} \BB_{pM_n (A)p}(pM_n (C)p\, , \, pM_n (A)p)$.

\begin{remark}\label{remark:MA1} We note that the unital inclusion of unital $C^*$-algebras
$A\subset C$ is strongly Morita equivalent to the unital inclusion of unital $C^*$-algebras
$pM_n (A)p\subset pM_n (C)p$ with respect to the $C-pM_n (C)p$-equivalence bimodule
$(1\otimes e)M_n (C)p$ and its closed subspace $(1\otimes e)M_n (A)p$, where we identify
$A$ and $C$ with $(1\otimes e)M_n (A)(1\otimes e)$ and $(1\otimes e)M_n (C)(1\otimes e)$,
respectively.
\end{remark}

\begin{lemma}\label{lem:MA2} With the above notation, let $\phi\in {}_A \BB_A (C, A)$. Then
for any $c\in M_n (C)$, $x, y\in M_n (A)$,
$$
{}_A \la (1\otimes e)xp\cdot F(\phi)(pcp) \, , \, (1\otimes e)zp \ra=\phi(\, {}_C \la (1\otimes e)xp\cdot pcp \, , \,
(1\otimes e)zp \ra ) .
$$
\end{lemma}
\begin{proof} This can be proved by routine computations. Indeed, for any $c\in M_n (C)$, $x, y\in M_n (A)$,
\begin{align*}
{}_A \la (1\otimes e)xp\cdot F(\phi)(pcp) \, , \, (1\otimes e)zp \ra &= 
{}_A \la (1\otimes e)xp(\phi\otimes\id)(pcp) \, , \, (1\otimes e)zp \ra \\
& =(1\otimes e)xp(\phi\otimes\id)(pcp)pz^* (1\otimes e) \\
& =(1\otimes e)xp(\phi\otimes\id)(c)pz^* (1\otimes e) .
\end{align*}
On the other hand,
$$
\phi(\, {}_C \la (1\otimes e)xp\cdot pcp \, , \, (1\otimes e)zp \ra)
=\phi((1\otimes e)xpcpz^* (1\otimes e)) .
$$
Since we identify $C$ with $(1\otimes e)M_n (C)(1\otimes e)$,
$$
\phi(\, {}_C \la (1\otimes e)xp\cdot pcp \, , \, (1\otimes e)zp \ra)
=(1\otimes e)xp(\phi\otimes\id)(c)pz^* (1\otimes e) .
$$
Thus, we obtain the conclusion.
\end{proof}

\begin{lemma}\label{lem:Ma3} With the above notation, for any $\phi\in {}_A \BB_A (C, A)$,
$$
f(\phi)=\Psi_B^{-1}\circ F(\phi)\circ\Psi_D .
$$
\end{lemma}
\begin{proof} By Lemma \ref{lem:con6}, it suffices to show that
$$
{}_A \la x\cdot (\Psi_B^{-1}\circ F(\phi)\circ\Psi_D )(d) \, , \, z \ra
=\phi(\, {}_C \la x\cdot d \, , \, z \ra)
$$
for any $d\in D$, $x, z\in X$. Indeed, for any $d\in D$, $x, z\in X$,
\begin{align*}
& {}_A \la x\cdot (\Psi_B^{-1}\circ F(\phi)\circ\Psi_D )(d) \, , \, z \ra \\
& =\,  {}_A \la \Psi_X (x\cdot (\Psi_B^{-1}\circ F(\phi)\circ\Psi_D)(d)) \, , \, \Psi_X (z) \ra \quad
(\text{by \cite [Lemma 2.6(3)]{KT4:morita}}) \\
& ={}_A \la \Psi_X (x)\cdot (F(\phi)\circ\Psi_D)(d) \, , \, \Psi_X (z) \ra \quad
(\text{by \cite [Lemma 2.6(2)]{KT4:morita}}) \\
& =\phi(\, {}_C \la \Psi_X (x)\cdot \Psi_D (d) \, , \, \Psi_X (z) \ra ) \quad
(\text{by Lemma \ref{lem:MA2}}) \\
& =\phi(\, {}_C \la \Psi_Y (x\cdot d) \, , \, \Psi_Y (z) \ra ) \quad
(\text{by \cite [Corollary 2.7(2), (5)]{KT4:morita}}) \\
& =\phi(\, {}_C \la x\cdot d \, , \, z \ra) \quad
(\text{by \cite [Corollary 2.7(3)]{KT4:morita}}) .
\end{align*}
Therefore, $f(\phi)=\Psi_B^{-1}\circ F(\phi)\circ\Psi_D$ for any $\phi\in {}_A \BB_A (C, A)$
by Lemma \ref{lem:con6}.
\end{proof}

\section{Basic properties}\label{sec:BP}
Let $A\subset C$ and $B\subset D$ be unital inclusions of unital $C^*$-algebras.
We suppose that they are strongly Morita equivalent with respect to a $C-D$-equivalence
bimodule $Y$ and its closed subspace $X$. Let ${}_A \BB_A (C, A)$ and ${}_B \BB_B (D, B)$
be as above and let $f$ be the isometric isomorphism of ${}_A \BB_A (C, A)$ onto
${}_B \BB_B (D, B)$ defined in Section \ref{sec:con}. In this section, we give basic properties about $f$.

\begin{lemma}\label{lem:BP1} With the above notation, we have the following:
\newline
$(1)$ For any selfadjoint linear map $\phi\in {}_A \BB_A (C, A)$, $f(\phi)$ is selfadjoint.
\newline
$(2)$ For any positive linear map $\phi\in{}_A \BB_A (C, A)$, $f(\phi)$ is positive.
\end{lemma}
\begin{proof} (1) Let $\phi$ be any selfadjoint linear map in ${}_A \BB_A (C, A)$ and
let $d\in D$, $x, z\in X$. By lemma \ref{lem:con6},
\begin{align*}
{}_A \la x\cdot f(\phi)(d^* ) \, , \, z \ra & =\phi( \, {}_C \la x\cdot d^* \, , \, z \ra )
=\phi( \, {}_C \la z\cdot d \, , \, x \ra^* ) \\
& =\phi( \, {}_C \la z\cdot d \, , \, x \ra )^* 
={}_A \la z \cdot f(\phi)(d) \, , \, x \ra^* \\
& ={}_A \la x\cdot f(\phi)(d)^* \, , \, z \ra .
\end{align*}
Hence $f(\phi)(d^* )=f(\phi)(d)^*$ for any $d\in D$.
(2) Let $d$ be any positive element in $D$.
Then ${}_C \la x\cdot d , x \ra\geq 0$ for any $x\in X$ by Raeburn and Williams
\cite [Lemma 2.28]{RW:continuous}. Hence $\phi( \, {}_C \la x\cdot d , x \ra)\geq 0$ for any $x\in X$. That is,
${}_A \la x\cdot f(\phi)(d)\, , \, x \ra \geq 0$ for any $x\in X$. Thus $f(\phi)(d)\geq 0$ by \cite
[Lemma 2.28]{RW:continuous}. Therefore, we obtain the conclusion.
\end{proof}

\begin{prop}\label{prop:BP3} If $\phi$ is a conditional expectation from $C$ onto $A$. Then
$f(\phi)$ is a conditional expectation from $D$ onto $B$.
\end{prop}
\begin{proof} Since $f(\phi)(a)=a$ for any $a\in A$, for any $b\in B$, $x, z\in X$,
$$
{}_A \la x\cdot f(\phi)(b) \, , \, z \ra =\phi(\, {}_C \la x\cdot b \, , \, z \ra)
=\phi(\, {}_A \la x\cdot b \, , \, z \ra)={}_A\la x\cdot b \, , \, z \ra
$$
by Lemma \ref{lem:con6}. Thus $f(\phi)(b)=b$ for any $b\in B$. By Proposition \ref{prop:con4} and
Lemma \ref{lem:BP1}, we obtain the conclusion.
\end{proof}

Since $A\subset C$ and $B\subset D$ are strongly Morita equivalent with respect to $Y$ and its closed
subset $X$, for any $n\in \BN$, $M_n (A)\subset M_n (C)$ and $M_n (B)\subset M_n (D)$ are
strongly Morita equivalent 
with respect to the $M_n (C)-M_n (D)$-equivalence bimodule $Y\otimes M_n (\BC)$ and its closed subspace
$X\otimes M_n (\BC)$, where we regard $M_n (\BC)$ as the trivial $M_n (\BC)-M_n (\BC)$-equivalence
bimodule. Let $f_n$ be the isometric isomorphism of ${}_{M_n (A)} \BB_{M_n (A)} (M_n (C), M_n (A))$ onto
${}_{M_n (B)} \BB_{M_n (B)}(M_n (D), M_n (B))$ defined in the same way as in the definition of
$f:{}_A \BB_A (C, A)\to {}_B \BB_B (D, B)$. Let $\phi\in {}_A \BB_A (C, A)$. Then
$$
\phi\otimes\id_{M_n (\BC)}\in {}_{M_n (A)} \BB_{M_n (A)} (M_n (C), M_n (A)) .
$$

\begin{lemma}\label{lem:BP4} With the above notation, let $n\in \BN$. Then for any $\phi\in {}_A \BB_A (C, A)$,
$$
f_n (\phi\otimes\id_{M_n (\BC)})=f(\phi)\otimes\id_{M_n (\BC)} .
$$
\end{lemma}
\begin{proof} This lemma can be proved by routine computations. Indeed, for any $d\in D$, $x, z\in X$,
$m_1, m_2, m_3 \in M_n (\BC)$,
\begin{align*}
& {}_{M_n (A)} \la [x\otimes m_1 \cdot f_n (\phi\otimes\id )(d\otimes m_2 )] \, , \, z\otimes m_3 \ra \\
& =(\phi\otimes\id)( \, {}_{M_n (C)} \la (x\otimes m_1 \cdot d\otimes m_2 )\, , \, z\otimes m_3 \ra ) \\
& =\phi( \, {}_C \la x\cdot d \, , \, z \ra )\otimes m_1 m_2 m_3^*  \\
& ={}_{M_n (A)} \la (x\otimes m_1 \cdot f(\phi)(d)\otimes m_2 ) \, , \, z\otimes m_3 \ra
\end{align*}
by Lemma \ref{lem:con6}.
Therefore, $f_n (\phi\otimes\id)=f(\phi)\otimes\id$ for any $\phi\in {}_A \BB_A (C, A)$ and $n\in\BN$.
\end{proof}

\begin{prop}\label{prop:BP5} With the above notation, let $\phi\in {}_A \BB_A (C, A)$. If $\phi$ is
$n$-positive, then $f(\phi)$ is $n$-positive for each $n\in \BN$.
\end{prop}
\begin{proof} This lemma is immediate by Lemmas \ref {lem:BP1}, \ref {lem:BP4}.
\end{proof}

\begin{prop}\label{prop} Let $\phi$ be a conditional expectation from $C$ onto $A$.
We suppose that there is a positive number $t$ such that
$$
\phi(c)\geq tc
$$
for any positive element $c\in C$. Then there is a positive number $s$ such that
$$
f(\phi)(d)\geq sd
$$
for any positive element $d\in D$.
\end{prop}
\begin{proof} We recall the discussions in Section \ref{sec:MA}. That is, by Lemma \ref{lem:MA2},
$f(\phi)=\Psi_B^{-1}\circ F(\phi)\circ\Psi_D$ for any $\phi\in {}_A \BB_A (C, A)$,
where $F$ is the isometric isomorphism of ${}_A \BB_A (C, A)$ onto
${}_{pM_n (A)p} \BB_{pM_n (A)p}(pM_n (C)p, pM_n (A)p)$, $n$ is some positive integer and $p$ is
a full projection in $M_n (A)$. Also, $\Psi_B$ and $\Psi_D$ are the isomorphisms of $B$ and $D$ onto
$pM_n (A)p$ and $pM_n (C)p$ defined in Section \ref{sec:MA}, respectively.
We note that $\Psi_D |_B =\Psi_B$. Since there is a positive integer $t$ such that $\phi(c)\geq tc$
for any positive element $c\in C$, by Frank and Kirchberg \cite [Theorem 1]{FK:conditional},
there is a positive number $s$ such that $(\phi\otimes\id)(c)\geq sc$ for any positive element $c\in M_n (C)$.
Thus for any $d\in D$,
\begin{align*}
f(\phi)(d^* d) & =(\Psi_B^{-1}\circ F(\phi)\circ\Psi_D )(d^* d)
=\Psi_B^{-1}((\phi\otimes\id)(\Psi_D (d)^* \Psi_D (d))) \\
& \geq \Psi_B^{-1}(s \Psi_D (d)^* \Psi_D (d)) =sd^* d
\end{align*}
since $F(\phi)=\phi\otimes\id$ and $\Psi_D |_B =\Psi_B$. Therefore, we obtain the conclusion.
\end{proof}

Following Watatani \cite [Definition 1.11.1]{Watatani:index}, we give the following definition.

\begin{Def}\label{def:BP6} Let $\phi\in{}_A \BB_A (C, A)$. Then a finite set $\{(u_i , v_i )\}_{i=1}^m \subset C\times C$
is called {\sl quasi-basis} for $\phi$ if it satisfies that
$$
c=\sum_{i=1}^m u_i \phi(v_i c)=\sum_{i=1}^m \phi(cu_i )v_i
$$
for any $c\in C$.
\end{Def}

\begin{lemma}\label{lem:BP7} With the above notation, let $\phi\in {}_A \BB_A (C, A)$ with a quasi-basis
$\{(u_i, v_i )\}_{i=1}^m$. Then there is a quasi-basis
$$
\{(p(u_i \otimes I_n )a_j p \, , \, pb_j (v_i \otimes I_n )p)\}_{i=1,2\dots, m, \, j=1,2\dots, K}
$$
for $F(\phi)$, where $a_1 , a_2 ,\dots a_K , b_1 , b_2 , \dots , b_K$ are elements in $M_n (A)$ with
$$
\sum_{j=1}^K a_j pb_j =1_{M_n (A)} .
$$
\end{lemma}
\begin{proof} This lemma can be proved in the same way as in \cite [Section 2]{KT4:morita}. Indeed,
$F(\phi)=(\phi\otimes\id_{M_n (\BC)})|_{pM_n (C)p}$, where $n$ is some positive integer and
$p$ is a full projection in $M_n (A)$. Hence there is elements
$a_1 , a_2 , \dots , a_K, b_1 , b_2 , \dots b_K \in M_n (A)$
such that $\sum_{i=1}^K a_i p b_i =1_{M_n (A)}$. Then the finite set
$$f
\{(p(u_i \otimes I_n )a_j p \, , \, pb_j (v_i \otimes I_n )p)\}_{i=1,2\dots, m, \, j=1,2\dots, K}
$$
is a quasi-basis for $F(\phi)$ by routine computations.
\end{proof}

\begin{prop}\label{prop:BP8} Let $\phi\in {}_A \BB_A (C, A)$. If there is a quasi-basis for $\phi$, then there is a
quasi-basis for $f(\phi)$.
\end{prop}
\begin{proof} This is immediate by Lemmas \ref{lem:Ma3}, \ref{lem:BP7}.
\end{proof}

\begin{cor}\label{cor:BP9} If $E^A$ is a conditional expectation from $C$ onto $A$, which is of Watatani index-finite
type, then $f(E^A )$ is a conditional expectation from $D$ onto $B$, which is of Watatani index-finite type.
\end{cor}
\begin{proof} This is immediate by Propositions \ref{prop:BP3}, \ref{prop:BP8}.
\end{proof}

\section{Modular automorphisms}\label{sec:Mod}
Following Watatani \cite [Section 1.11]{Watatani:index}, we give the definitions of the modular condition and
modular automorphisms.
\par
Let $A\subset C$ be a unital inclusion of unital $C^*$-algebras. Let $\theta$ be an automorphism of $A' \cap C$ and
$\phi$ an element in ${}_A \BB_A (C, A)$.

\begin{Def}\label{def:Mod1} Let $\theta$ and $\phi$ be as above. $\phi$ is said to satisfy the {\sl modular condition}
for $\theta$ if the following condition holds:
$$
\phi(xy)=\phi(y\theta(x))
$$
for any $x\in A' \cap C$, $y\in C$.
\end{Def}
We have the following theorem which was proved by Watatani in \cite {Watatani:index}:

\begin{thm}\label{thm:Mod2} {\rm (cf: \cite [Theorem 1.11.3]{Watatani:index})} Let $\phi\in {}_A \BB_A (C, A)$
and we suppose that there is a quasi-basis for $\phi$. Then there is the unique automorphism $\theta$ of
$A' \cap C$ for which $\phi$ satisfies the modular condition.
\end{thm}

\begin{Def}\label{def:Mod3} The above automorphism $\theta$ of $A' \cap C$ given by $\phi$ in
Theorem \ref{thm:Mod2} is called the {\sl modular automorphism} associated with $\phi$ and denoted by
$\theta^{\phi}$.
\end{Def}

\begin{remark}\label{remark:Mod4} Following the proof of \cite [Theorem 1.11.3]{Watatani:index},
we give how to construct the modular automorphism $\theta^{\phi}$. Let $\{(u_i , v_i )\}_{i=1}^m$ be
a quasi-basis for $\phi$. Put
$$
\theta^{\phi}(c)=\sum_{i=1}^m u_i \phi(cv_i )
$$
for any $c\in A' \cap C$. Then $\theta^{\phi}$ is the unique automorphism of $A' \cap C$ satisfying the
modular condition by the proof of \cite [Theorem 1.11.3]{Watatani:index}.
\end{remark}

Let $A\subset C$ and $B\subset D$ be unital inclusions of unital $C^*$-algebras, which are strongly
Morita equivalent. Then by Section \ref{sec:MA},
$$
B\cong pM_n (A)p \, , \quad D\cong pM_n (C)p ,
$$
where $n$ is some positive integer and $p$ is a full projection in $M_n (A)$.
Also, there is the isometric isomorphism
$$
F: {}_A \BB_A (C, A)\longrightarrow {}_{pM_n (A)p} \BB_{pM_n (A)p} (pM_n (C)p \, ,\, pM_n (A)p) ,
$$
which is defined in Section
\ref{sec:MA}. Furthermore, by the proof of \cite [Lemma 10.3]{KT4:morita}, there is the isomorphism $\pi$
of $A' \cap C$ onto $(pM_n (A)p)' \cap pM_n (C)p$ defined by
$$
\pi(c)=(c\otimes I_n )p
$$
for any $c\in A' \cap C$, where we note that $(pM_n (A)p)' \cap pM_n (C)p=(M_n (A)' \cap M_n (C))p$.
and that
$$
M_n (A)' \cap M_n (C)=\{c\otimes I_n \, | \, c\in A' \cap C \}.
$$
Thus we can see that
$$
\pi^{-1}((c\otimes I_n )p)=\sum_{j=1}^K a_j (c\otimes I_n )pb_j =c\otimes I_n
$$
for any $c\in A' \cap C$, where $a_1 , a_2 , \dots , a_K , b_1 , b_2 , \dots , b_K$
are elements in $M_n (A)$ with $\sum_{j=1}^K a_j pb_j =1_{M_n (A)}$ and we identify $M_n (A)' \cap M_n (C)$
with $A' \cap C$ by the isomorphism
$$
A' \cap C \to M_n (A)' \cap M_n (C) : c\to c\otimes I_n .
$$

\begin{lemma}\label{lem:Mod5} With the above notation, let $\phi\in {}_A \BB_A (C, A)$
with a quasi-basis for $\phi$. Then $F(\phi)\in {}_{pM_n (A)p} \BB_{pM_n(A)p}(pM_n (C)p \, , \, pM_n (A)p)$
with a quasi-basis for $F(\phi)$ and
$$
\theta^{F(\phi)}=\pi\circ\theta^{\phi}\circ\pi^{-1} .
$$
\end{lemma}
\begin{proof} Let $\{(u_i , v_i )\}_{i=1}^m$ be a quasi-basis for $\phi$. Then by Lemma \ref{lem:BP7},
$$
\{(p(u_i \otimes I_n )a_j p \, , \, pb_j (v_i \otimes I_n )p)\}_{i=1,2\dots, m, \, j=1,2\dots, K}
$$
is a quasi-basis for $F(\phi)$, where $a_1 , a_2 , \dots , a_K , b_1 , b_2 , \dots , b_K$
are elements in $M_n (A)$
with $\sum_{j=1}^K a_j pb_j =1_{M_n (A)}$.
Then by Remark \ref{remark:Mod4} and the definitions of $F(\phi)$, $\phi$,
for any $c\in A' \cap C$,
\begin{align*}
\theta^{F(\phi)}((c\otimes I_n )p) & =\sum_{i, j}p(u_i \otimes I_n )a_j pF(\phi)((c\otimes I_n )pb_j (v_i \otimes I_n )p) \\
& =\sum_{i, j}p(u_i \otimes I_n )a_j p(\phi\otimes\id)((c\otimes I_n )pb_j (v_i \otimes I_n )p) \\
& =\sum_{i, j}p(u_i \otimes I_n )(\phi\otimes \id)(a_j pb_j (cv_i \otimes I_n ))p \\
& =\sum_i p(u_i \otimes I_n )(\phi\otimes\id)(cv_i \otimes I_n )p \\
& =\sum_i p(u_i \phi(cv_i )\otimes I_n ) \\
& =(\theta^{\phi}(c)\otimes I_n )p .
\end{align*}
On the other hand, for any $c\in A' \cap C$,
$$
(\pi\circ\theta^{\phi}\circ\pi^{-1})((c\otimes I_n )p)=(\pi\circ\theta^{\phi})(c)=(\theta^{\phi}(c)\otimes I_n )p .
$$
Hence
$$
\theta^{F(\phi)}(c)=(\pi\circ\theta^{\phi}\circ\pi^{-1})(c)
$$
for any $c\in(pM_n (A)p)' \cap pM_n (C)p$. Therefore, we obtain the conclusion.
\end{proof}

\begin{thm}\label{thm:Mod6} Let $A\subset C$ and $B\subset D$ be unital inclusions of unital $C^*$-algebras
which are strongly Morita equivalent. Let $\phi$ be any element in ${}_A \BB_A (C, A)$ with a quasi-basis for
$\phi$. Let $f$ be the isometric isomorphism of ${}_A \BB_A (C, A)$ onto ${}_B \BB_B (D, B)$ defined in
Section \ref{sec:con}. Then $f(\phi)$ is an element in ${}_B \BB_B (D, B)$ with a quasi-basis for $f(\phi)$ and
there is an isomorphism $\rho$ of $A' \cap C$ onto $B' \cap D$ such that
$$
\theta^{f(\phi)}=\rho\circ\theta^{\phi}\circ\rho^{-1} .
$$
\end{thm}
\begin{proof} By Proposition \ref{prop:BP8}, $f(\phi)\in {}_B \BB_B (D, B)$ with a quasi-basis for $f(\phi)$.
Also, by Lemma \ref{lem:Mod5},
$$
\theta^{F(\phi)}=\pi\circ\theta^{\phi}\circ\pi^{-1} ,
$$
where $\pi$ is the isomorphism of $A' \cap C$ onto $(pM_n (A)p)' \cap pM_n (C)p$ defined as above and $n$ is
some positive integer, $p$ is a full projection in $M_n(A)$. Let $\Psi_D$ be the isomorphism of $D$ onto
$pM_n (C)p$ defined in Section \ref{sec:MA}. Since $\Psi_D |_B$ is an isomorphism of $B$ onto $pM_n (A)p$,
$\Psi_D^{-1}\circ\theta^{F(\phi)}\circ\Psi_D$ can be regarded as an automorphism of $B' \cap D$.
We claim that $f(\phi)$ satisfies the modular condition for $\Psi_D^{-1}\circ\theta^{F(\phi)}\circ\Psi_D$.
Indeed, by Lemma \ref{lem:Ma3}, for any $x\in B' \cap D$, $y\in D$,
\begin{align*}
f(\phi)(xy) & =(\Psi_B^{-1}\circ F(\phi)\circ\Psi_D )(xy) \\
& =(\Psi_B^{-1}\circ F(\phi))(\Psi_D (x)\Psi_D (y)) \\
& =(\Psi_B^{-1}\circ F(\phi))(\Psi_D (y)\theta^{F(\phi)}(\Psi_D (x))) \\
& =(\Psi_B^{-1}\circ F(\phi)\circ\Psi_D )(y(\Psi_D^{-1}\circ\theta^{F(\phi)}\circ\Psi_D )(x)) \\
& =f(\phi)(y(\Psi_D^{-1}\circ\theta^{F(\phi)}\circ\Psi_D )(x)) .
\end{align*}
Hence by Theorem \ref{thm:Mod2}, $\theta^{f(\phi)}=\Psi_D^{-1}\circ\theta^{F(\phi)}\circ\Psi_D$.
Thus, we obtain the conclusion.
\end{proof}

We give a remark. Following \cite [Section 1.4]{Watatani:index}, for any $\phi\in {}_A \BB_A (C, A)$,
$h\in A' \cap C$, we define $\phi_h$ and ${}_h \phi$ as follows: Let $\phi_h$ and ${}_h \phi$ be the linear
maps from $C$ to $A$ defined by
$$
\phi_h (c)=\phi(hc) \, , \quad {}_h \phi (c)=\phi(ch)
$$
for any $c\in C$, respectively. Then by easy computations, $\phi_h$ and ${}_h \phi$ are in ${}_A \BB_A (C, A)$.

\begin{remark}\label{remark:Mod7} Let $A\subset C$ and $B\subset D$ be unital inclusions of unital
$C^*$-algebras, which are strongly Morita equivalent. Let $f$, $F$ and $\pi$, $\Psi_B$, $\Psi_D$
be as in the proof of Theorem \ref{thm:Mod6}.
Then we have the following:
\newline
(1) $F(\phi_h )=F(\phi)_{\pi(h)} \, , \quad f(\phi_h )=f(\phi)_{(\Psi_D^{-1}\circ\pi)(h)}$,
\newline
(2) $F({}_h \phi)={}_{\pi(h)}F(\phi) \, , \quad f({}_h \phi)=\, {}_{(\Psi_D^{-1}\circ\pi)(h)}f(\phi)$,
\newline
for any $\phi\in{}_A \BB_A (C, A)$, $h\in A' \cap C$.
We show (1). For any $c\in pM_n (C)p$,
\begin{align*}
F(\phi_h )(c) & =(\phi_h \otimes\id)(c)
=(\phi\otimes\id)((h\otimes I_n )c)
=(\phi\otimes\id)_{(h\otimes I_n )p}(c) \\
& =F(\phi)_{\pi(h)}(c)
\end{align*}
by the definition of $F$. Hence $F(\phi_h )=F(\phi)_{\pi(h)}$. Also, for any $d\in D$,
\begin{align*}
f(\phi_h )(d) & =(\Psi_B^{-1}\circ F(\phi_h )\circ\Psi_D )(d)=\Psi_B^{-1}(F(\phi_h)(\Psi_D (d))) \\
& =\Psi_B^{-1}(F(\phi)_{\pi(h)}(\Psi_D (d)))=\Psi_B^{-1}(F(\phi)(\pi(h)\Psi_D (d))) \\
& =(\Psi_B^{-1}\circ F(\phi)\circ\Psi_D )((\Psi_D^{-1}\circ\pi)(h)d) \\
& =f(\phi)((\Psi_D^{-1}\circ\pi)(h)d) \\
& =f(\phi)_{(\Psi_D^{-1}\circ\pi)(h)}(d)
\end{align*}
by the above discussion and Lemma \ref{lem:Ma3}, where $n$ is some positive integer and $p$ is
a full projection in $M_n (A)$. Hence $f(\phi_h )=f(\phi)_{(\Psi_D^{-1}\circ\pi)(h)}$.
Similarly, we can show (2).
\end{remark}

\section{Equivalence classes}\label{sec:EC} Let $A\subset C$ and $B\subset D$ be unital inclusions
of unital $C^*$-algebras. We suppose that they are strongly Morita equivalent with respect to
a $C-D$-equivalence bimodule $Y$ and its closed subspace $X$. Let ${}_A \BB_A (C, A)$ and ${}_B \BB_B (D, B)$
be as in Section \ref{sec:con} and let $f_{(X, Y)}$ be the isometric isomorphism of ${}_A \BB_A (C, A)$
onto ${}_B \BB_B (D, B)$ induced by $(X, Y)$, which is
defined in Section \ref{sec:con}. Then $f_{(X, Y)}$ depends on the choice of a
$C-D$-equivalence bimodule $Y$ and its closed subspace $X$. In this section, we shall clarify the relation
between equivalence classes of $C-D$-equivalence bimodules $Y$ and their closed subspaces $X$ and
isometric isomorphisms of ${}_A \BB_A (C, A)$ onto ${}_B \BB_B (D, B)$.
\par
Let $\Equi(A, C, B, D)$ be the set of all pairs $(X, Y)$ such that $Y$ is a $C-D$-equivalence bimodule and
$X$ is its closed subspace satisfying Conditions (1) and (2) in \cite [Definition 2.1]{KT4:morita}.
We define an equivalence relation $``\sim"$ in $\Equi(A, C, B, D)$ as follows:
For any $(X, Y), (Z, W)\in\Equi(A, C, B, D)$, we say that $(X, Y)\sim(Z, W)$ in $\Equi(A, C, B, D)$ if
there is a $C-D$-equivalence bimodule isomorphism $\Phi$ of
$Y$ onto $W$ such that $\Phi|_X$ is a bijection of $X$ onto $Z$. Then $\Phi|_X$ is
an $A-B$-equivalence bimodule isomorphism of $X$ onto $Z$ by \cite [Lemma 3.2]{Kodaka:Picard2}.
We denote by $[X, Y]$ the equivalence class of $(X, Y)\in\Equi(A, C, B, D)$.

\begin{lemma}\label{lem:EC1} With the above notation, let $(X, Y), (Z, W)\in\Equi(A, C, B, D)$ with
$(X, Y)\sim(Z, W)$ in $\Equi(A, C, B, D)$. Then $f_{(X, Y)}=f_{(Z, W)}$.
\end{lemma}
\begin{proof} Let $\Phi$ be a $C-D$-equivalence bimodule of $Y$ onto $W$ satisfying that
$\Phi|_X$ is an $A-B$-equivalence bimodule isomorphism of $X$ onto $Z$. Let $\phi\in {}_A \BB_A (C, A)$.
Then for any $x_1 , x_2 \in X$, $d\in D$,
\begin{align*}
{}_A \la \Phi(x_1 )\cdot f(_{(Z, W)}(\phi))(d) \, , \, \Phi(x_2 ) \ra & =\phi(\, {}_C \la \Phi(x_1 )\cdot d \, , \, 
\Phi(x_2 ) \ra) \\
& =\phi(\, {}_C \la \Phi(x_1 \cdot d) \, , \, \Phi(x_2 ) \ra ) \\
& =\phi( \, {}_C \la x_1 \cdot d \, , \, x_2 \ra ) \\
& ={}_A \la x_1 \cdot f_{(X, Y)}(\phi)(d) \, , \, x_2 \ra .
\end{align*}
On the other hand,
\begin{align*}
{}_A \la \Phi(x_1 )\cdot f_{(Z, W)}(\phi)(d) \, , \, \Phi(x_2 ) \ra & ={}_A \la \Phi(x_1 \cdot f_{(Z, W)}(\phi)(d)) \, , \,
\Phi(x_2 ) \ra \\
& ={}_A \la x_2 \cdot f_{(Z, W)}(\phi)(d) \, , \, x_2 \ra .
\end{align*}
Hence we obtain that
$$
{}_A \la x_1 \cdot f_{(X, Y)}(\phi)(d) \, , \, x_2 \ra ={}_A \la x_2 \cdot f_{(Z, W)}(\phi)(d) \, , \, x_2 \ra
$$
for any $x_1, x_2 \in X$. Thus we obtain the conclusion.
\end{proof}

We denote by $f_{[X, Y]}$ the isometric
isomorphism of ${}_A \BB_A (C, A)$ onto ${}_B \BB_B(D, B)$ 
induced by the equivalence class $[X, Y]$ of an element $(X, Y)\in\Equi(A, C, B, D)$.
\par
Let $L\subset M$ be a unital inclusion of unital $C^*$-algebras, which is strongly Morita equivalent
to the unital inclusion $B\subset D$ with respect to a $D-M$-equivalence bimodule $W$ and
its closed subspace $Z$. Then the inclusion $A\subset C$ is strongly Morita equivalent to the inclusion
$L\subset M$ with respect to the $C-M$-equivalence bimodule $Y\otimes_D W$ and its closed subspace
$X\otimes_B Z$.

\begin{thm}\label{thm:EC2} Let $A\subset C$, $B\subset D$ and $L\subset M$ be unital inclusions of unital
$C^*$-algebras. We suppose that $A\subset C$ and $B\subset D$ are strongly Morita with respect to
a $C-D$ equivalence bimodule $Y$ and its closed subspace $X$ and that $B\subset D$ and $L\subset M$
are strongly Morita equivalent with respect to a $D-M$-equivalence bimodule $W$ and its closed subspace
$Z$. Let $f_{[X, Y]}$ and $f_{[Z, W]}$ be the isomorphic isomorphisms of ${}_A \BB_A (C, A)$ and
${}_B \BB (D, B)$ onto ${}_B \BB_B (D, B)$ and ${}_L \BB_L (M, L)$ induced by the equivalence classes
$[X, Y]$ and $[Z, W]$, respectively. Then
$$
f_{[X\otimes_B Z \, , \, Y\otimes_D W]}=f_{[Z, W]}\circ f_{[X, Y]} ,
$$
where $f_{[X\otimes_B Z \, , \, Y\otimes_D W]}$ is the isometric isomorphism of ${}_A \BB_A (C, A)$
onto ${}_L \BB_L (M, L)$ induced by the equivalence class $[X\otimes_B Z \, , \, Y\otimes_D W]$ of
$(X\otimes_B Z \, , \, Y\otimes_D W)\in \Equi(A, C, L, M)$.
\end{thm}
\begin{proof} Let $x_1 , x_2 \in X$ and $z_1 , z_2 \in Z$. Let $m\in M$ and $\phi\in {}_A \BB_A (C, A)$.
We note that $f_{[Z, W]}(f_{[X, Y]}(\phi))(m)\in L$. Hence
\begin{align*}
& {}_A \la x_1 \otimes z_1 \cdot f_{[Z, W]}(f_{[X, Y]}(\phi))(m) \, , \, x_2 \otimes z_2 \ra \\
& ={}_A \la x_1 \cdot {}_B \la z_1 \cdot f_{[Z, W]}(f_{[X, Y]}(\phi))(m) \, , \, z_2 \ra \, , \, x_2 \ra \\
& ={}_A \la x_1 \cdot f_{[X, Y]}(\phi)({}_D \la z_1 \cdot m \, , \, z_2 \ra ) \, , \, x_2 \ra \\
& =\phi(\, {}_C \la x_1 \cdot {}_D \la z_1 \cdot m \, , \, z_2 \ra \, , \, x_2 \ra .
\end{align*}
On the other hand,
\begin{align*}
& {}_A \la x_1 \otimes z_1 \cdot f_{[X\otimes_B Z \, , \, Y\otimes_D W]}(\phi)(m) \, , \, x_2 \otimes z_2 \ra \\
& =\phi(\, {}_C \la x_1 \otimes z_1 \cdot m \, , \, x_2 \otimes z_2 \ra ) \\
& =\phi(\, {}_C \la x_1 \cdot {}_D \la z_1 \cdot m \, , \, z_2 \ra \, , \, x_2 \ra ) .
\end{align*}
By Lemma \ref{lem:con6},
$$
f_{[Z, W]}(f_{[X, Y]}(\phi))=f_{[X\otimes_B Z \, , \, Y\otimes_D W]}(\phi)
$$
for any $\phi\in {}_A \BB_A (C, A)$. Therefore, we obtain the conclusion.
\end{proof}


\begin{thebibliography}{99}

%\bibitem{AEE:crossed}B. Abadie, S. Eilers and R. Exel,
%{\it Morita equivalence for crossed products by Hilbert $C^*$-bimodules},
%Trans. Amer. Math. Soc.,
%{\bf 350}
%(1998),
%3043--3054.

%\bibitem{Blackadar:K-Theory}B. Blackadar, {\it K-theory for operator algebras},
%M. S. R. I. Publications 5, 2nd Edition,
%Cambridge Univ. Press, Cambridge, 1998.

%\bibitem{BC:discrete}E. B\'edos and R. Conti,
%{\it On discrete twisted $C^*$-dynamical systems, Hilbert $C^*$-modules and regularity},
%preprint, arXiv: 1104.173lv1.

%\bibitem{BCM:crossed}R. J. Blattner, M. Cohen and S. Montgomery,
%{\it Crossed products and inner actions of Hopf algebras},
%Trans. Amer. Math. Soc.,
%{\bf 298}
%(1986),
%671--711.

%\bibitem{Brown:hereditary}L. G. Brown,
%{\it Stable isomorphism of hereditary subalgebra of $C^*$-algebras},
%Pacific J. Math.,
%{\bf 71}
%(1977),
%335--348.

%\bibitem{BGR:linking}L. G. Brown, P. Green and M. A. Rieffel,
%{\it Stable isomorphism and strong Morita equivalence of $C^*$-algebras},
%Pacific J. Math.,
%{\bf 71}
%(1977),
%349--363.

%\bibitem{BMS:quasi}L. G. Brown, J. Mingo and N-T. Shen,
%{\it Quasi-multipliers and embeddings of Hilbert $C^*$-bimodules},
%Can. J. Math.,
%{\bf 46}
%(1994),
%1150--1174.

%\bibitem{CK:outer} M. Choda and H. Kosaki,
%{\it Strongly outer actions for an inclusion of factors,}
%J. Func. Anal.
%{\bf 122}
%(1994), 315-332.

%\bibitem{Combes:morita}F. Combes,
%{\it Crossed products and Morita equivalence},
%Proc. London Math. Soc.,
%{\bf 49}
%(1984),
%289--306.

%\bibitem{CMW:equivalence}R. E. Curto, P. S. Muhly and D. P. Williams,
%{\it Cross products of strong Morita equivalent $C^*$-algebras},
%Proc. Amer. Math. Soc.,
%{\bf 90}
%(1984),
%528--530.

%\bibitem{ER:multiplier}S. Echterhoff and I. Raeburn,
%{\it Multipliers of imprimitivity bimodules and Morita equivalence of
%crossed products},
%Math. Scand.,
%{\bf 76}
%(1995), 289--309.

\bibitem{FK:conditional} M. Frank and E. Kirchberg,
{ \it On  conditional expectations of finite index},
J. Operator Theory,
{\bf 40}
(1998), 87-111.

%\bibitem{JW:covariant}S. Jansen and S. Waldmann,
%{\it The H-covariant strong Picard groupoid},
%J. Pure  Appl. Algebra, 
%{\bf 205}
%(2006), 542--598.

%\bibitem{JT:KK}K. K. Jensen and K. Thomsen,
%{\it Elements of KK-theory},
%Birkh$\ddot a$user,
%1991.

%\bibitem{Izumi:simple}M. Izumi,
%{\it Inclusions of simple $C^*$-algebras},
%J. reine angew. Math.,
%{\bf 547}
%(2002), 97--138.

%\bibitem{KW1:bimodule}T. Kajiwara and Y. Watatani,
%{\it Jones index theory by Hilbert $C^*$-bimodules and K-Thorey},
%Trans. Amer. Math. Soc.,
%{\bf 352}
%(2000), 3429--3472.

%\bibitem{KW2:discrete}T. Kajiwara and Y. Watatani,
%{\it Crossed products of Hilbert $C^*$-bimodules by countable discrete groups},
%Proc. Amer. Math. Soc.,
%{\bf 126}
%(1998), 841--851.

%\bibitem{Kodaka:Picard}K. Kodaka,
%{\it Picard groups of irrational rotation $C^*$-algebras},
%J. London Math. Soc.,
%{\bf 56}
%(1997), 179-188.

%\bibitem{Kodaka:equivariance}K. Kodaka,
%{\it Equivariant Picard groups of $C^*$-algebras with finite dimensional $C^*$-Hopf algebra coactions},
%Rocky Mountain J. Math.,
%{\bf 47}
%(2017), 1565-1615.

%\bibitem{Kodaka:generalized}K. Kodaka,
%{\it The generalized Picard groups for finite dimensional $C^*$-Hopf algebra coactions
%on unital $C^*$-algebras},
%preprint, arXiv:1805.08358.

\bibitem{Kodaka:Picard2}K. Kodaka,
{\it The Picard groups for unital inclusions of unital $C^*$-algebras}, preprint,
arXiv: 1712.09499v1, Acta. Sci. Math. (Szeged), to appear.

%\bibitem{Kodaka:bundle}K. Kodaka,
%{\it Equivalence bundles over a finite group and strong Morita equivalence for unital inclusions
%of unital $C^*$-algebras}, preprint,
%arXiv: 1905.10001.


%\bibitem{Kodaka:countable}K. Kodaka,
%{\it Strong Morita equivalence for inclusions of $C^*$-algebras induced by twisted actions of
%a countable discrete group}, preprint.

%\bibitem{Kodaka:Picard3}K. Kodaka,
%{\it The Picard groups of unital inclusions of unital $C^*$-algebras induced by involutive
%equivalence bundles},
%preprint.

%\bibitem{KT1:inclusion}K. Kodaka and T. Teruya,
%{\it Inclusions of unital $C^*$-algebras of index-finite type with depth 2 induced by saturated
%actions of finite dimensional $C^*$-Hopf algebras},
%Math. Scand.,
%{\bf 104}
%(2009),
%221--248.

%\bibitem{KT2:coaction}K. Kodaka and T. Teruya,
%{\it The Rohlin property for coactions of finite dimensional $C^*$-Hopf algebras on
%unital $C^*$-algebras},
%J. Operator Theory,
%{\bf 74}
%(2015),
%329--369.

%\bibitem{KT3:equivalence}K. Kodaka and T. Teruya,
%{\it The strong Morita equivalence for coactions of a finite dimensional $C^*$-Hopf algebra on
%unital $C^*$-algebras},
%Studia Math.,
%{\bf 228}
%(2015),
%259--294.

\bibitem{KT4:morita}K. Kodaka and T. Teruya,
{\it The strong Morita equivalence for inclusions of $C^*$-algebras and conditional expectations for
equivalence bimodules}, J. Aust. Math. Soc.,
{\bf 105}
(2018), 103--144.

%\bibitem{KT5:inclusion2}K. Kodaka and T. Teruya,
%{\it Coactions of a finite dimensional $C^*$-Hopf algebra on unital
%$C^*$-algebras, unital inclusions of unital $C^*$-algebras and the strong Morita equivalence},
%preprint, arXiv:1706.09430.

%\bibitem{KT6:free} K. Kodaka and T. Teruya,
%{\it Free coactions of a finite dimensional $C^*$-Hopf algebra and strong Morita equivalence}, preprint.

%\bibitem{Lance:toolkit}E. C. Lance, {\it Hilbert $C^*$-modules},
%A toolkit for operatros algebraists, London Math. Soc., Lecture Note Series,
%{\bf 210}, 
%Cambridge Univ. Press, Cambridge, 1995.

%\bibitem{MT:Kac}T. Masuda and R. Tomatsu,
%{\it Classification of minimal actions of a compact Kac algebra with the amenable dual},
%J. Funct. Anal.,
%{\bf 258}
%(2010),
%1965-2025.

%\bibitem{OKT:rohlininclusion}H. Osaka, K. Kodaka and T. Teruya,
%{\it The Rohlin property for inclusions of $C^*$-algebras with a finite Watatani index},
%Operator structures and dynamical systems, 177--195, Contemp Math.,
%{\bf 503}
%%Amer Math. Soc., Providence, RI, 2009.

%\bibitem{Packer:projective}J. A. Packer,
%{\it $C^*$-algebras generated by projective representations of the discrete Heisenberg group},
%J. Operator Theory,
%{\bf 18}
%(1987),
%41--66.

%\bibitem{Pedersen:auto} G. K. Pedersen,
%{\it $C^* $-algebras and their automorphism groups,}
%Academic Press, London, New York, Sn Francisco, 1979.

\bibitem{RW:continuous}I. Raeburn and D. P. Williams,
{\it Morita equivalence and continuous -trace $C^*$-algebras},
Mathematical Surveys and Monographs, {\bf 60}, Amer. Math. Soc., 1998.

\bibitem{Rieffel:rotation}M. A. Rieffel,
{\it $C^*$-algebras associated with irrational rotations},
Pacific J. Math.,
{\bf 93}
(1981),
415--429.

\bibitem{Stormer:positive} E. St\o rmer,
{\it Positive linear maps of operator algebras},
Springer-Verlag, Berlin, Heidelberg, 2013.

%\bibitem{Sweedler:Hopf}
%%M. E. Sweedler,
%{\it Hopf algebras},
%Benjamin, New York, 1969.

%\bibitem{Szymanski:subfactor}W. Szyma\'nski,
%{\it Finite index subfactors and Hopf algebra crossed products},
%Proc. Amer. Math. Soc.,
%{\bf 120}
%(1994),
%519--528.

%\bibitem{SP:saturated}W. Szyma\'nski and C. Peligrad,
%{\it Saturated actions of finite dimensional Hopf {\rm *}-algebras on  
%$C^*$-algebras},
%Math. Scand.,
%{\bf 75}
%(1994),
%217--239.

%\bibitem{SW:compact} K. Schwieger and S. Wagner,
%{\it Free actions of compact groups on $C^*$-algebras, Part I}, preprint,
%arXiv: 1505,00688v1.

%\bibitem{Tomiyama:projection}J. Tomiyama,
%{\it On the projection of norm one in $W^*$-algebras},
%Japan Acad.,
%{\bf 33}
%(1957),
%608--612.

\bibitem{Watatani:index}Y. Watatani,
{\it Index for $C^*$-subalgebras},
Mem. Amer. Math. Soc.,
{\bf 424}, Amer. Math. Soc.,
1990.

%\bibitem{Zarikian:expectation}V. Zarikian,
%{\it Unique expectations for discrete crossed products}, preprint, arXiv: 170709339v1.

\end{thebibliography}
\end{document}